\documentclass[12pt]{article}%
\usepackage{amsmath}
\usepackage{amsfonts}
\usepackage{amssymb}
\usepackage{graphicx}%
 \makeatletter
\newfont{\footsc}{cmcsc10 at 8truept}
\newfont{\footbf}{cmbx10 at 8truept}
\newfont{\footrm}{cmr10 at 10truept}
\newtheorem{theorem}{Theorem}

\newtheorem{proposition}[theorem]{Proposition}

\newenvironment{proof}[1][Proof]{\noindent{\textbf {#1}  }}  {\hfill$\Box$\bigskip}

\def\blfootnote{\xdef\@thefnmark{}\@footnotetext}
\begin{document}

\title{Spectra of the blow-up graphs}
\author{
Carla Oliveira\thanks{National School of Statistics, Rio de Janeiro, Brazil;\textit{email:
carla.oliveira@ibge.gov.br}} 
\thanks{Research supported by CNPq Grant 305454/2012-9.} $\;\;$
Leonardo de Lima\thanks{Department of Production Engineering, Federal
Center of Technological Education, Rio de Janeiro, Brazil;\textit{email:
llima@cefet-rj.br}} 
\thanks{Research supported by Programa Jovem Cientista do Nosso Estado, FAPERJ Grant E-26/102.218/2013, and CNPq Grant 305867/2012-1.} \ 
and$\;$
Vladimir Nikiforov\thanks{Department of Mathematical Sciences, University of Memphis, Memphis TN 38152, USA; \textit{email: vnikifrv@memphis.edu}} }
\date{}
\maketitle

\begin{abstract}
Let $G$ be graph on $n$ vertices and $G^{(t)}$ its blow-up graph of order $t.$ In this paper, we determine all eigenvalues of the Laplacian and the signless Laplacian matrix of $G^{(t)}$ and its complement $\overline{G^{(t)}}.$
\end{abstract}

\section{Introduction}

\label{intro}

Let $G$ be a graph on $n$ vertices with degree sequence $d_{1},d_{2}%
,\ldots,d_{n}$ in non-incresing order. We denote
the complement of $G$ by $\overline{G}$ $.$ Write $A=A(G)$ for the adjacency
matrix of $G$ and let $D=D(G)$ be the diagonal matrix of the row-sums of $A$,
i.e., the degrees of $G$. The Laplacian $L(G)$ and signless Laplacian $Q(G)$
of $G$ are defined as $L(G)=D-A$ and $Q(G)=D+A$. The eigenvalues of $A(G)$,
$L(G)$ and $Q(G)$ arranged in non-increasing order are denoted by $\lambda
_{1},\ldots,\lambda_{n}$, $\mu_{1},\ldots,\mu_{n}$ and $q_{1},\ldots,q_{n}.$

For any graph $G$ and integer $t\geq1,$ write $G^{\left(  t\right)  }$ for the
graph obtained by replacing each vertex $u$ of $G$ by a set $V_{u}$ of $t$
independent vertices and every edge $\left\{  u,v\right\}  $ of $G$ by a
complete bipartite graph with parts $V_{u}$ and $V_{v}.$ Usually $G^{\left(
t\right)  }$ is called a \textit{blow-up} of $G.$ See Figure \ref{fig:blowup}
for a blow-up example of order $2$ and $3$. \begin{figure}[h]
\begin{center}
\includegraphics{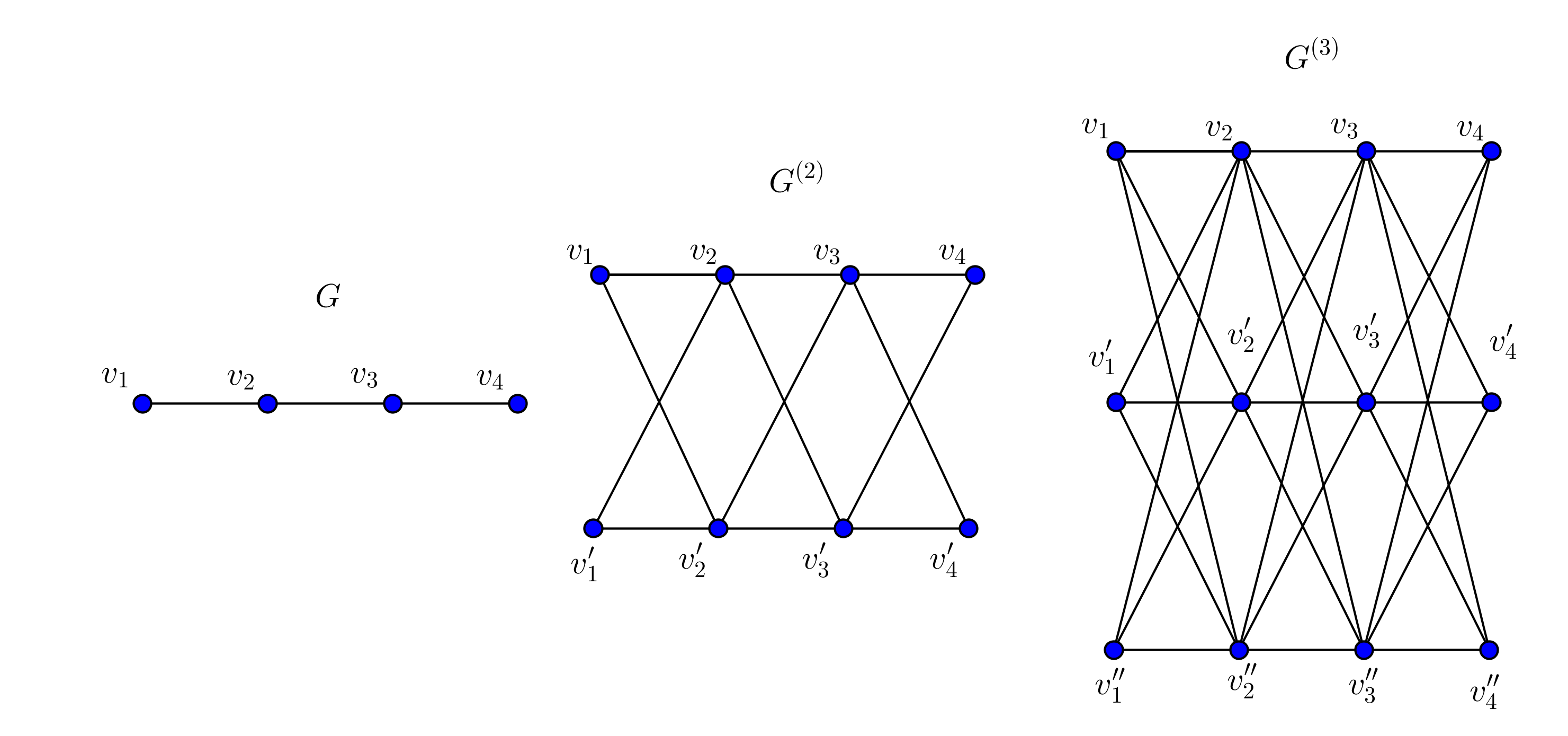}
\end{center}
\par
\label{fig:blowup}\caption{A graph $G$ and its blow-up of order $2$ and $3$.}%
\end{figure}

The blow-up of a graph $G$ has a useful algebraic characterization: if $A$ is
the adjacency matrix of $G,$ then the adjacency matrix $A\left(  G^{\left(
t\right)  }\right)  $ of $G^{\left(  t\right)  }$ is given by
\[
A\left(  G^{\left(  t\right)  }\right)  =A\otimes J_{t}%
\]
where $\otimes$ is the Kronecker product and $J_{t}$ is the all ones matrix of
order $t.$ This observation yields the following fact.

\begin{proposition}
\label{pro1}The\ eigenvalues of $G^{\left(  t\right)  }$ are $t\lambda_{1}\left(
G\right)  ,\ldots,t \lambda_{n}\left(  G\right)  ,$ together with $n\left(
t-1\right)  $ additional $0$'s.
\end{proposition}

For the complements of graph blow-ups, one can easily check the following fact.

\begin{proposition}
\label{pro2}The\ eigenvalues of $\overline{G^{\left(  t\right)  }}$ are
$t\lambda_{1}\left(  \overline{G}\right)  +t-1,\ldots,t\lambda_{n}\left(  \overline
{G}\right)  +t-1,$ together with $n\left(  t-1\right)  $ additional $\left(
-1\right)  $'s.
\end{proposition}

The goal of this paper is to determine the spectra of $L(G^{(t)})$,
$Q(G^{(t)}),$ and $Q(\overline{G^{(t)}}),$ which is a more difficult task than
for the adjacency matrix$.$ We also note that the spectrum of $L(\overline
{G^{(t)}})$ is obtained immediately from our results as $\mu_{n-i}%
(\overline{G})=n-\mu_{i}(G),$ for $i=1,\ldots,n-1,$ (see e.g.
\cite{AM85}).

\section{Main Results}

\label{main}

In this section, we prove the main results of the paper.

\begin{theorem}\label{thm1}
If $G$ is a graph on $n$ vertices, the Laplacian eigenvalues of $G^{(t)}$ are $t{\mu}_1, \ldots,t{\mu}_{n}$ and $td_1,\ldots,td_n$, where each $td_i$ has multiplicity $(t-1)$ for $i=1, \ldots,n$.
\end{theorem}
\begin{proof}
The Laplacian matrix of $G^{(t)}$ can be written as 
{\small{$$ L(G^{(t)})=
\left[ \begin{array}{ccccc} L(G)+(t-1)D &-A(G)&\ldots&&-A(G)\\
-A(G)&L(G)+(t-1)D&\ldots&&-A(G)\\
-A(G)&-A(G)&\ldots&&-A(G)\\ 
\vdots & & \ddots && \vdots \\ 
-A(G)&-A(G)&\ldots &&L(G)+(t-1)D
\end{array}\right],$$ }}
where $L(G)+(t-1)D$ is the matrix of order $n$ and $D$ is the diagonal matrix of the vertices degree.
Let $\{\mathbf{x}_{i}\}$ be an orthogonal basis of eigenvectors to $L(G)$, such that for each $i=1, \ldots, n,$ $\mathbf{x}_{i}$  is an eigenvector to ${\mu}_i.$
Consider the $tn-$vector $\mathbf{y}_{i}$ such that ${\mathbf{y}_{i}}^T = [ \mathbf{x}_{i} \;\; \mathbf{x}_{i} \;\; \cdots \;\; \mathbf{x}_{i}],$
%$\mathbf{y}_{i} = \left[\begin{array}{c} \mathbf{x}_{i} \\ \hline \mathbf{x}_{i} \\ \hline  \vdots \\ \hline \mathbf{x}_{i} \end{array}\right],$
for $i=1, \ldots, n.$ Observe that 
{\small{$$ L(G^{(t)})\mathbf{y}_i = 
\left[\begin{array}{c} L(G) \mathbf{x}_{i} + (t-1)D \mathbf{x}_{i} - (t-1)A(G)\mathbf{x}_{i} \\ 
\hline L(G) \mathbf{x}_{i} + (t-1)D \mathbf{x}_{i} - (t-1)A(G) \mathbf{x}_{i} \\ \hline  \vdots \\ \hline L(G) \mathbf{x}_{i} + (t-1)D \mathbf{x}_{i} - (t-1)A(G)\mathbf{x}_{i} \end{array}\right]= \left[\begin{array}{c} tL(G) \mathbf{x}_{i} \\ \hline tL(G) \mathbf{x}_{i} \\ \hline  \vdots \\ \hline tL(G) \mathbf{x}_{i} \end{array}\right] = \left[\begin{array}{c} t{\mu}_i \mathbf{x}_{i} \\ \hline t{\mu}_i \mathbf{x}_{i} \\ \hline  \vdots \\ \hline t{\mu}_i \mathbf{x}_{i} \end{array}\right]= t \mu_{i} \mathbf{y}_{i}.$$}} So $\mathbf{y}_{i}$ is an eigenvector to $t \mu_{i}.$ It easy to see that $\mathbf{y}_{i}^T \mathbf{y}_{j} = 0,$ since $\mathbf{x}_{i}^T \mathbf{x}_{j} = 0,$ for all $i \neq j$, $1 \leq i,j \leq n.$
Consider the $tn$-vector $E_i$ of the form ${E_i}^T = [{\alpha}_{i1} \;\;  {\alpha}_{i2} \;\; \cdots \;\; {\alpha}_{it}]$
%$E_i = \left[\begin{array}{c} {\alpha}_{i1}  \\ \hline  {\alpha}_{i2} \\ \hline \vdots \\ \hline {\alpha}_{it} \end{array}\right]$
such that, ${\alpha}_{ik} \in {\Re}^n$ for all $1 \leq k \leq t$. For each $i=1, \ldots, n$, let $e_i$ the $i-th$ standard unit basis vector.
For each $k=1,\ldots, t-1$, we define the $nt$-vector
$${({E_i}^k)}^T  = [ 0 \;\; \cdots \;\; 0 \;\; e_{ik} \;\; -e_{i(k+1)} \;\; 0 \;\; \cdots \;\; 0]$$
%${E_i}^k  = \left[\begin{array}{c} 0 \\ \hline \vdots  \\ \hline 0 \\ \hline e_{ik} \\ \hline -e_{i(k+1)} \\ \hline 0 \\ \hline \vdots \\ \hline 0 \end{array}\right] $
and we get  $L(G^{(t)}){E_i}^k = td_i {E_i}^k.$
So ${E_i}^1, {E_i}^2, \ldots, {E_i}^{t-1}$ are eigenvectors to $td_i.$ Now, we need to prove that the set $\{{E_i}^1, {E_i}^2, \ldots, {E_i}^{t-1}\}$ is linearly independent. Suppose that $\sum_{j=1}^{t-1} c_j {E_i}^{j}=0$. So $c_1=c_{t-1}=0$ and $-c_j+c_{j+1}=0$, for all $j=1,\ldots,t-2.$ Then $c_j=0$, $1 \leq j \leq t-1$ which implies that $td_{i}$ is an eigenvalue of multiplicity at least $t-1.$ Also, $\mathbf{x}_{j}^{T} E_{i}^{k}=0$ for $i \neq j$, $1 \leq i,j \leq n$ and $1 \leq k \leq t-1$ and the result follows.
\end{proof}

\begin{theorem}\label{thm2}
If $G$ be is a graph on $n$ vertices, the signless Laplacian eigenvalues of $G^{(t)}$ are $t q_1, \ldots, t q_{n}$ and
$td_1,\ldots,td_n$, where each $td_i$ has multiplicity $(t-1)$ for $i=1, \ldots,n$.
\end{theorem}
\begin{proof}
The signless Laplacian matrix of $G^{(t)}$ can be written as 
{\small{$$ Q(G^{(t)})=
\left[ \begin{array}{ccccc} Q(G)+(t-1)D & A(G)& \ldots&& A(G)\\ 
A(G)&Q(G)+(t-1)D& \ldots&& A(G)\\ 
A(G) & A(G)&\ldots&& A(G)\\ 
\vdots & & \ddots && \vdots \\  
A(G) & A(G)&\ldots & & Q(G)+(t-1)D
\end{array}\right],$$ }}
where $Q(G)+(t-1)D$ is the matrix of order $n$ and $D$ is the diagonal matrix of the vertices degree.
The proof follows analogously to Theorem \ref{thm1} taking the same eigenvectors ${\mathbf{y}_{i}}^T =[\mathbf{x}_{i} \;\; \mathbf{x}_{i} \;\; \cdots \;\; \mathbf{x}_{i}],$
%$\mathbf{y}_{i} =\left[\begin{array}{c} \mathbf{x}_{i} \\ \hline \mathbf{x}_{i} \\ \hline  \vdots \\ \hline \mathbf{x}_{i} \end{array}\right],$  where $\mathbf{x}_{i}$
$\mathbf{y}_{i}$ is an eigenvector to $q_{i}$ for $i=1, \ldots, n,$ and also
$${({E_i}^k)}^T  = [0 \;\; \cdots \;\; 0 \;\; e_{ik} \;\; -e_{i(k+1)} \;\; 0 \;\; \cdots \;\; 0]$$
%${E_i}^k  = \left[\begin{array}{c} 0 \\ \hline \vdots  \\ \hline 0 \\ \hline e_{ik} \\ \hline -e_{i(k+1)} \\ \hline 0 \\ \hline \vdots \\ \hline 0 \end{array}\right] $
for each $i=1,\ldots,n$ and  $k=1,\ldots, t-1,$ and the result follows.
\end{proof}

%Given a graph $G$ and an integer $t>0$, set $G^{[t]}= \overline{{\overline{G}^{(t)}}}$, i.e., $G^{[t]}$ is obtained from $G^{(t)}$ by joining all vertices within $V_u$ for every vertex $u$ of $G$.

%\begin{theorem}\label{thm3}
%Let $G$ be a graph on $n$ vertices and $G^{(t)}$ the blow-up of $G.$ Then, the signless Laplacian eigenvalues of $G^{[t]}$ are $t q_1 + 2(t-1), \ldots, t q_{n} + 2(t-1)$ and
%$t(d_1+1) -2,\ldots,t(d_n+1) -2$, where each $t(d_i+1)-2$ has multiplicity $(t-1)$ for $i=1, \ldots,n$.
%\end{theorem}
%\begin{proof}
%The signless Laplacian matrix of $G^{[t]}$ can be written as {\footnotesize{$$ Q(G^{(t)})=\left[ \begin{array}{cccccc} Q(G)+(t-1)(D+I) & A(G)+I& A(G)+I&\ldots
%&& A(G)+I\\ A(G)+I&Q(G)+(t-1)D& A(G)+I&\ldots&& A(G)+I\\ A(G)+I & A(G)+I&Q(G)+(t-1)(D+I)&\ldots&& A(G)+I\\ \vdots & & & \ddots && \vdots \\  A(G)+I & A(G)+I&\ldots & A(G)+I && Q(G)+(t-1)(D+I)
%\end{array}\right],$$ }}
%where $Q(G)+(t-1)(D+I)$ is the matrix of order $n$ and $D$ is the diagonal matrix of the vertices degree. The proof follows analogously to Theorem \ref{thm1} taking the same eigenvectors $\mathbf{y}_{i} =\left[\begin{array}{c} \mathbf{x}_{i} \\ \hline \mathbf{x}_{i} \\ \hline  \vdots \\ \hline \mathbf{x}_{i} \end{array}\right],$  where $\mathbf{x}_{i}$ is an eigenvector to the eigenvalue $q_{i}$ for $i=1, \ldots, n,$ and also
%${E_i}^k  = \left[\begin{array}{c} 0 \\ \hline \vdots  \\ \hline 0 \\ \hline e_{ik} \\ \hline -e_{i(k+1)} \\ \hline 0 \\ \hline \vdots \\ \hline 0 \end{array}\right] $ for each $i=1,\ldots,n$ and  $k=1,\ldots, t-1.$
%\end{proof}

\begin{theorem}\label{thm4}
If $G$ is a graph on $n$ vertices, the signless Laplacian eigenvalues of $\overline{G^{(t)}}$ are $t \overline{q}_1 + 2(t-1), \ldots, t \overline{q}_{n} + 2(t-1)$ and
$tn -td_1-2,\ldots,tn -td_n-2$, where each $tn -td_i -2$ has multiplicity $(t-1)$ for $i=1, \ldots,n$.
\end{theorem}
\begin{proof}
With a convenient labeling of the blow-up graph $G^{(t)},$ we can write the signless Laplacian of the complement of $G^{(t)}$ as the following
{\footnotesize{$$ Q(\overline{G^{(t)}})=
\left[ \begin{array}{cccc}
(tn-2)I_{n}-tD+J_{n}-A & \ldots &  J_{n} - A(G)\\
J_{n} - A(G)  & \ddots & J_{n} - A(G)\\
\vdots       &   \vdots                 & \vdots \\
J_{n} - A(G) & \cdots &  (tn-2)I_{n}-tD+J_{n}-A
\end{array}\right].$$ }}
%For each $i=1,\ldots,n,$ let $\mathbf{y}_{i}$ be the $tn-$vector defined as
%$$\mathbf{y} = \left(
%\begin{array}{c}
%e_{i}  \\ \hline
%-e_{i}  \\ \hline
%0 \\
%\vdots  \\
%0
%\end{array}\right).$$
Considering the $tn-$vectors ${E_i}^k$ for each $k=1,\ldots,n$ we get
$$Q(\overline{G^{(t)}}) {E_i}^k = \left(
\begin{array}{c}
0 \;\; \ldots \;\; (tn-2-td_{i})  \;\; \ldots \;\; 0  \;\; 0 \;\; \ldots \;\; -(tn-2-td_{i})  \;\; \ldots \;\; 0 \;\; 0 \;\; \ldots  \;\; 0
\end{array}\right)^{T}$$
and then $Q(\overline{G^{(t)}}){E_i}^k = (tn-2-td_i) {E_i}^k.$
So ${E_i}^1, {E_i}^2, \ldots, {E_i}^{t-1}$ are eigenvectors to $tn-2-td_i$ for all $i=1,\ldots,n.$
Since the set $\{{E_i}^1, {E_i}^2, \ldots, {E_i}^{t-1}\}$ is linearly independent, the eigenvalue $tn-2-td_i$ has multiplicity at least $t-1.$ Therefore, $nt-n$ eigenvalues of $Q(\overline{G^{(t)}})$ are known and we need to find the remaining $n.$
One can easily rewrite $Q(\overline{G^{(t)}})$ in a such way that the block matrix $Q(\overline{G})+(t-1)nI_n-(t-1)D$ appears in the main diagonal and $A(\overline{G})+I_n$ in the remaining positions.
Let $\{\mathbf{x}_{i}\}$ be an orthogonal basis of eigenvectors to $Q(\overline{G})$, such that for each $i=1, \ldots, n,$ $\mathbf{x}_{i}$  is an eigenvector to ${\overline{q}}_i.$
Consider the $tn-$vectors $\mathbf{y}_{i}$ such that ${\mathbf{y}_{i}}^T = [ \mathbf{x}_{i} \;\; \mathbf{x}_{i} \;\; \cdots \;\; \mathbf{x}_{i}],$
%$\mathbf{y}_{i}$ such that $\mathbf{y}_{i} = \left[\begin{array}{c} \mathbf{x}_{i} \\ \hline \mathbf{x}_{i} \\ \hline  \vdots \\ \hline \mathbf{x}_{i} \end{array}\right],$
for $i=1, \ldots, n.$ Observe that
{\small{$$ Q(\overline{G^{(t)}})\mathbf{y}_i = \left[\begin{array}{c} Q(\overline{G}) \mathbf{x}_{i} + (t-1)n \mathbf{x}_{i} - (t-1)D\mathbf{x}_{i} + (t-1)(A(\overline{G})+I_n)\mathbf{x}_{i}\\ \hline Q(\overline{G}) \mathbf{x}_{i} + (t-1)n \mathbf{x}_{i} - (t-1)D\mathbf{x}_{i} + (t-1)(A(\overline{G})+I_n)\mathbf{x}_{i} \\ \hline  \vdots \\ \hline Q(\overline{G}) \mathbf{x}_{i} + (t-1)n \mathbf{x}_{i} - (t-1)D\mathbf{x}_{i} + (t-1)(A(\overline{G})+I_n)\mathbf{x}_{i} \end{array}\right]=$$}}
{\small{$$= \left[\begin{array}{c} {\overline{q}}_i \mathbf{x}_{i} + (t-1)(n-1)I_n \mathbf{x}_{i} - (t-1)D\mathbf{x}_{i} + (t-1)A(\overline{G})\mathbf{x}_{i} + 2(t-1)\mathbf{x}_{i} \\ \hline {\overline{q}}_i \mathbf{x}_{i} + (t-1)(n-1)I_n \mathbf{x}_{i} - (t-1)D\mathbf{x}_{i} + (t-1)A(\overline{G})\mathbf{x}_{i} + 2(t-1)\mathbf{x}_{i}  \\ \hline  \vdots \\ \hline  {\overline{q}}_i \mathbf{x}_{i} + (t-1)(n-1)I_n \mathbf{x}_{i} - (t-1)D\mathbf{x}_{i} + (t-1)A(\overline{G})\mathbf{x}_{i} + 2(t-1)\mathbf{x}_{i} \end{array}\right]=$$}}
{\small{$$  = \left[\begin{array}{c} \overline{q_i} \mathbf{x}_{i} + ((t-1)(n-1)I_n - (t-1)D) \mathbf{x}_{i} + (t-1)A(\overline{G})\mathbf{x}_{i} + 2(t-1)\mathbf{x}_{i}\\ \hline \overline{q_i} \mathbf{x}_{i} + ((t-1)(n-1)I_n - (t-1)D) \mathbf{x}_{i} + (t-1)A(\overline{G})\mathbf{x}_{i} + 2(t-1)\mathbf{x}_{i} \\ \hline  \vdots \\ \hline \overline{q_i} \mathbf{x}_{i} + ((t-1)(n-1)I_n - (t-1)D) \mathbf{x}_{i} + (t-1)A(\overline{G})\mathbf{x}_{i} + 2(t-1)\mathbf{x}_{i} \end{array}\right]=$$}}
{\small{$$\left[\begin{array}{c} {\overline{q}}_i \mathbf{x}_{i} + (t-1)D(\overline{G})\mathbf{x}_{i} + (t-1)A(\overline{G})\mathbf{x}_{i} + 2(t-1)\mathbf{x}_{i} \\ \hline {\overline{q}}_i \mathbf{x}_{i} + (t-1)D(\overline{G})\mathbf{x}_{i} + (t-1)A(\overline{G})\mathbf{x}_{i} + 2(t-1)\mathbf{x}_{i}  \\ \hline  \vdots \\ \hline  {\overline{q}}_i \mathbf{x}_{i} + (t-1)D(\overline{G})\mathbf{x}_{i} + (t-1)A(\overline{G})\mathbf{x}_{i} + 2(t-1)\mathbf{x}_{i} \end{array}\right]=$$}}
{\small{$$  = \left[\begin{array}{c} \overline{q_i} \mathbf{x}_{i} + (t-1)Q(\overline{G}) \mathbf{x}_{i} + 2(t-1)\mathbf{x}_{i}\\ \hline \overline{q_i} \mathbf{x}_{i} + (t-1)Q(\overline{G}) \mathbf{x}_{i} + 2(t-1)\mathbf{x}_{i} \\ \hline  \vdots \\ \hline \overline{q_i} \mathbf{x}_{i} + (t-1)Q(\overline{G}) \mathbf{x}_{i} + 2(t-1)\mathbf{x}_{i} \end{array}\right]= \left[\begin{array}{c} t{\overline{q}}_i \mathbf{x}_{i} + 2(t-1)\mathbf{x}_{i} \\ \hline t{\overline{q}}_i \mathbf{x}_{i} + 2(t-1)\mathbf{x}_{i}\\ \hline  \vdots \\ \hline  t{\overline{q}}_i \mathbf{x}_{i} + 2(t-1)\mathbf{x}_{i} \end{array}\right]= (t {\overline{q}}_{i}+ 2(t-1)) \mathbf{y}_{i}.$$}}
So $\mathbf{y}_{i}$ is an eigenvector to $t {\overline{q}}_{i}+ 2(t-1).$ It easy to see that $\mathbf{y}_{i}^T \mathbf{y}_{j} = 0,$ since $\mathbf{x}_{i}^T \mathbf{x}_{j} = 0,$ for all $i \neq j$, $1 \leq i,j \leq n$ and the result follows.
\end{proof}

%Given a graph $G$ and an integer $t>0$, set $G^{[t]}= \overline{{\overline{G}^{(t)}}}$, i.e., $G^{[t]}$ is obtained from $G^{(t)}$ by joining all vertices within $V_u$ for every vertex $u$ of $G$.

%\begin{theorem}\label{thm5}
%Let $G$ be a graph on $n$ vertices and $G^{(t)}$ the blow-up of $G.$ Then, the signless Laplacian eigenvalues of $G^{[t]}$ are $t q_1 + 2(t-1), \ldots, t q_{n} + 2(t-1)$ and
%$t(d_1+1) -2,\ldots,t(d_n+1) -2$, where each $t(d_i+1)-2$ has multiplicity $(t-1)$ for $i=1, \ldots,n$.
%\end{theorem}
%\begin{proof}
%The signless Laplacian matrix of $G^{[t]}$ can be written as {\footnotesize{$$ Q(G^{(t)})=\left[ \begin{array}{cccccc} Q(G)+(t-1)(D+I) & A(G)+I& A(G)+I&\ldots
%&& A(G)+I\\ A(G)+I&Q(G)+(t-1)D& A(G)+I&\ldots&& A(G)+I\\ A(G)+I & A(G)+I&Q(G)+(t-1)(D+I)&\ldots&& A(G)+I\\ \vdots & & & \ddots && \vdots \\  A(G)+I & A(G)+I&\ldots & A(G)+I && Q(G)+(t-1)(D+I)
%\end{array}\right],$$ }}
%where $Q(G)+(t-1)(D+I)$ is the matrix of order $n$ and $D$ is the diagonal matrix of the vertices degree. The proof follows analogously to Theorem \ref{thm1} taking the same eigenvectors $\mathbf{y}_{i} =\left[\begin{array}{c} \mathbf{x}_{i} \\ \hline \mathbf{x}_{i} \\ \hline  \vdots \\ \hline \mathbf{x}_{i} \end{array}\right],$  where $\mathbf{x}_{i}$ is an eigenvector 
% to the eigenvalue $q_{i}$ for $i=1, \ldots, n,$ and also
%${E_i}^k  = \left[\begin{array}{c} 0 \\ \hline \vdots  \\ \hline 0 \\ \hline e_{ik} \\ \hline -e_{i(k+1)} \\ \hline 0 \\ \hline \vdots \\ \hline 0 \end{array}\right] $ for each $i=1,\ldots,n$ and  $k=1,\ldots, t-1.$
%\end{proof}

\end{document}